\documentclass[12pt]{article}

\usepackage[english]{babel}
\usepackage[utf8]{inputenc}
\usepackage[T1]{fontenc}
\usepackage{lmodern}
\usepackage{amssymb,amsmath,amsfonts}
\usepackage{xspace}
\usepackage{graphicx}
\usepackage{hyperref}
\usepackage{url}
\usepackage{stmaryrd}
\usepackage[babel=true]{csquotes}
\usepackage{amsthm}
\usepackage{subfigure}
\usepackage{mathrsfs}
\usepackage{float}

\usepackage{tikz}

\usepackage{geometry}

\theoremstyle{plain}
\newtheorem{thm}{Theorem}[section]
\newtheorem{lm}[thm]{Lemma}
\newtheorem{prop}[thm]{Proposition}

\theoremstyle{definition}
\newtheorem{dfn}[thm]{Definition}
\newtheorem{rmk}[thm]{Remark}

\def\N{\mathbb{N}\xspace}
\def\Z{\mathbb{Z}\xspace}

\def\R{\mathbb{R}\xspace}
\def\C{\mathbb{C}\xspace}

\hypersetup{pdfborder={0000}, colorlinks=true, linkcolor=blue, citecolor=red}

\newcommand{\om}{\omega} 
\newcommand{\al}{\alpha} 
\newcommand{\be}{\beta}
\newcommand{\ga}{\gamma}

\newcommand{\si}{\sigma}
 
\newcommand{\Om}{\Omega}
\newcommand{\ep}{\epsilon}

\newcommand{\Ci}{\mathcal{C}^{\infty}} 
\newcommand{\con}{\overline}
\newcommand{\bigo}{\mathcal{O}}

\newcommand{\wt}{\widetilde}
\newcommand{\ti}{\tilde}

\newcommand{\op}{\operatorname}

\newcommand{\id}{\op{id}}

\newcommand{\Hilb}{\mathcal{H}}

\begin{document}

\title{Analytic Berezin-Toeplitz operators}

\author{Laurent Charles}

\maketitle

\abstract{We introduce new tools for analytic microlocal analysis on
  K\"ahler manifolds. As an application, we prove that the space of Berezin-Toeplitz operators with
analytic contravariant symbol is an algebra. We also give a short proof of the
Bergman kernel asymptotics up to an exponentially small error. }

\section{Introduction}

Asymptotic behavior of Bergman kernels on K\"ahler manifolds has been studied in
many papers after the pioneer works of Bouche \cite{Bouche}, Tian \cite{Tian}, Zeldtich \cite{Ze} and Catlin \cite{Ca}. New asymptotics with exponentially small error terms
have been obtained recently by Rouby-Sj\"ostrand-Vu Ngoc \cite{RoSjVu}. In
the same vein, Deleporte \cite{Del} started to develop analytic microlocal techniques
for Berezin-Toeplitz operators on K\"ahler manifolds.

Our goal in this paper is to further develop this program in two ways. First, we introduce new methods to simplify the
complicated estimates of \cite{Del} and obtain a short proof of the
 Bergman kernel estimates of \cite{RoSjVu}. Second, we define the algebra of 
 Berezin-Toeplitz operators whose contravariant symbols are analytic symbols
 in the sense of \cite{SjAs}. 

Before we present our main results and methods, let us start with two typical
expansions in this theory: the Bergman kernel on the diagonal and the
composition of two Toeplitz operators. In both cases, we will compare the
(smooth) usual version of the result with its (analytic) improvement.

Let $M$ be a compact complex manifold equipped with two holomorphic Hermitian
line bundles $L$, $L'$. Assume $L$ is positive and for any $k \in \N$, let
$\Hilb_k$ be the space of holomorphic section of $L^k \otimes L'$.
The Bergman kernel of $L^k \otimes L'$ is defined by $\Pi_k(x, \con y) =
\sum_{i=1}^{d_k} \Psi_{i} (x) \otimes \con \Psi_i (y)$, for any $x,y \in M$,   where $(\Psi_i)_{i=1, \ldots , d_k}$ is any orthonormal basis of $\Hilb_k$.
It was proved in \cite{Ca}, \cite{Ze} from \cite{BoSj} that the Bergman kernel
has the following asymptotic expansion on the diagonal
\begin{gather} \label{eq:asexp_diagonal_bergman}
  \Pi_{k} (x, \con x) = \biggl( \frac{k}{2 \pi} \biggl)^n \sum_{\ell=0}^N
  k^{-\ell} \rho_{\ell} (x) + \bigo ( k^{-N-1}), \qquad \forall N
\end{gather}
with smooth coefficients $\rho_{\ell} \in \Ci(M)$. A lot can be said on
these coefficients but our aim here is to improve the remainder. Let us assume from now on that the
metrics of $L$ and $L'$ are analytic. Then if we replace the finite sum  by a
partial sum over all integers $\ell$ smaller than $\ep k$ with $\ep$
sufficiently small, the remainder
becomes exponentially small. The precise result, proved in \cite{RoSjVu}, is that there exist $\ep>0$ and $C>0$ such that 
\begin{gather} \label{eq:asexp_diagonal_bergman_an}
  \Pi_{k} (x, \con x) = \biggl(\frac{k}{2 \pi} \biggr)^n  \; \sum_{\ell=0}^{\lfloor
    \ep k \rfloor}  k^{-\ell} \rho_{\ell} (x) + \bigo ( e^{-k/C})
\end{gather}
with a $\bigo$ uniform on $M$. The case of surfaces with constant curvature
was done before by Berman \cite{Be12}.

Let us consider now Toeplitz operators. For any  $f \in \Ci (M)$, let  $T_k(f)
: \Hilb_k \rightarrow \Hilb_k$ be the operator sending $\Psi$ into $\Pi_k ( f
\Psi)$ with $\Pi_k$  the orthogonal projection of $\Ci ( M , L^k \otimes L')$
onto $\Hilb_k$. By Boutet de Monvel-Guillemin \cite{BoGu}, \cite{Gu3},  for any $f,g \in \Ci (M)$, 
\begin{gather} \label{eq:toep_prod_inf}
  T_k (f) T_k (g) = \sum_{\ell =0}^{N} k^{-\ell}  T_k (h_{\ell}) + \bigo ( k^{-N-1}), \qquad \forall N
\end{gather}
for some coefficients $h_{\ell} \in \Ci (M)$. In the case $f$ and $g$ are analytic, we will prove that there exist $\ep>0$ and $C>0$ such that 
\begin{gather} \label{eq:toep_prod_exp}
  T_k (f) T_k (g) = \sum_{\ell =0}^{\lfloor \ep k \rfloor} k^{-\ell}  T_k (h_{\ell})  + \bigo ( e^{-k/C})
\end{gather}
where the $\bigo$ is in uniform norm. 

These two results, Bergman kernel expansion and Toeplitz composition, belong
actually to the same theory. Let us explain this first in the usual smooth setting
and then in the analytic case. 
Define a Berezin-Toeplitz operator as any family $(S_k) \in \prod_{k\geqslant
  1} \op{End} (\Hilb_k) $ such that
\begin{gather} \label{eq:def_BT}
  S_k = \sum_{\ell =0}^{N} k^{-\ell}  T_k (f_{\ell}) + \bigo ( k^{-N-1}), \qquad \forall N
\end{gather}
for a sequence $(f_{\ell})$ of $\Ci (M)$. By \cite{Gu3}, \cite{BoGu}, the space $\mathcal{T}$ of
Berezin-Toeplitz operator is a subalgebra of $\prod_k \op{End} \Hilb_k$. In \cite{oim}, we proved the following
characterization of the Schwartz kernel of Berezin-Toeplitz operators. Here,
the Schwartz kernel of $S_k \in \op{End} ( \Hilb_k)$ is the holomorphic section of $(L^k
\otimes L') \boxtimes (\con L^k \otimes \con L ')$ given by
$$ S_k (x, \con y ) = \sum_{i =1} ^{d_k} (S_k \Psi_i) (x) \otimes \con{\Psi}_i
  (y), \qquad x,y \in M $$ 
  where $(\Psi_i)_{i=1, \ldots , d_k}$ is any orthonormal basis of $\Hilb_k$.
  Then $(S_k) \in \prod
\op{End}( \Hilb_k)$ is a Berezin-Toeplitz operator if and only if for any
compact set $K$ not intersecting the diagonal $|S_k (x,y_c) | = \bigo (
k^{-N})$ on $K$ for any $N$ and on a neighborhood $W $ of the diagonal  
\begin{gather} \label{eq:kernel_toep_intro}
S_k (x,y_c) = \biggl( \frac{k}{2 \pi } \biggr)^n E^k (x,y_c) E' (x,y_c)
\sum_{\ell= 0 }^N k^{-\ell} \tilde{\tau}_\ell (x,y_c) + \bigo ( k^{n-N-1}) 
\end{gather}
for any $N$, 
where $E$, $E'$ are holomorphic sections of $L\boxtimes \con L $ and $L'
\boxtimes \con L'$ on $W$ such that their restrictions to the diagonal are the
canonical sections determined by the metrics, and the $\ti \tau_{\ell}$'s are smooth
functions on $W$ such that $\con \partial \ti \tau_{\ell}$ vanishes to infinite
order along the diagonal.

In particular, since the Bergman kernel is the Schwartz kernel of the
Berezin-Toeplitz operator $(\op{id}_{\Hilb_k})$,  $\Pi_k(x,y_c)$ satisfies the
expansion (\ref{eq:kernel_toep_intro}), which extends the expansion 
\eqref{eq:asexp_diagonal_bergman} outside the diagonal. 

For the analytic version, we need the notion of analytic symbol \cite{SjAs}. We say that a
formal series $\sum \hbar^\ell a_{\ell}$ of $\Ci ( M ) [[\hbar]]$ is an {\em analytic
symbol} if there exist a neighborhood $W \subset M \times \con M$ of the diagonal
and $C>0$ such that each $a_{\ell}$ has a holomorphic extension $\ti a_{\ell}$ to
$W$ satisfying $|\ti a_{\ell} |\leqslant C^{\ell +1} \ell !$. When $\ep
<1/C$, we set
\begin{gather} \label{eq:partial_sum_an_symbol}
a ( \ep, k ) = \sum_{\ell =0 }^{ \lfloor \ep k \rfloor} a_\ell
k^{-\ell}, \qquad \ti a ( \ep, k ) = \sum_{\ell =0 }^{ \lfloor \ep k \rfloor} \ti a_\ell
k^{-\ell}.
\end{gather}
We have already used these partial sums in \eqref{eq:asexp_diagonal_bergman} and \eqref{eq:toep_prod_inf}. As a result, 
 the corresponding series $\sum \hbar^{\ell}
\rho_{\ell}$ and $\sum \hbar^{\ell} h_{\ell}$  are analytic
symbols.

An {\em analytic} Berezin-Toeplitz operator is a family $(S_k ) \in \prod_k
\op{End} \Hilb_k $ such that
\begin{gather}  \label{eq:def_an_ber_toep}
S_k = T_k ( f(\ep, k ) ) + \bigo ( e^{-k/C})
\end{gather}
for an analytic symbol $\sum \hbar^\ell f_\ell$ and $C>0$. 

\begin{thm} $ $ \label{th:intro}
  \begin{enumerate} 
    \item The space of   analytic Berezin-Toeplitz
      operator is a subalgebra of $\mathcal{T}$.
      \item A family $ (S_k) \in \prod_k \op{End} ( \Hilb_k)$ is an analytic
        Berezin-Toeplitz operator if and only if for any compact set $K$ of $M
        \times \con M$ not intersecting the diagonal $|S_k ( x,y_c) | = \bigo
        ( e^{-k/C_K})$ on $K$ for some $C_K >0$ and on a neighborhood of the
        diagonal
\begin{gather*}
        S_k (x,y_c) = \Biggl( \frac{k}{2 \pi } \Biggr)^n E^k (x,y_c) E' (x,y_c)
\tilde{\tau} (\ep, k )  (x,y_c) + \bigo ( e^{-k/C} ) 
\end{gather*}
where $E$, $E'$ are the same sections as in \eqref{eq:kernel_toep_intro},
$\ti \tau (\ep, k)$ is the partial sum associated to an analytic symbol $\sum
\hbar^{\ell} \tau_{\ell}$ as in \eqref{eq:partial_sum_an_symbol} and $C>0$.
\end{enumerate}
\end{thm}

\begin{rmk} $ $
  \begin{enumerate}
  \item
    By the first assertion, $(\op{id}_{\Hilb_k})$ is an analytic
Berezin-Toeplitz operator and thus the second assertion describes the Bergman
kernel up to $\bigo ( e^{-k/C})$. This description was the main result of
\cite[Theorem 6.1]{RoSjVu}.
\item In \cite{Del}, the operators $(S_k) \in \prod_k \op{End} \Hilb_k$ whose
  Schwartz kernel satisfies the condition given in the second assertion are
  called covariant Toeplitz operators. The main results of \cite{Del} are first that these operators are closed under product and second that the
  operators with an elliptic symbol have an inverse in the same class of
  operators.
  \item Given the previous remarks, the original result in Theorem
    \ref{th:intro} is its second assertion:  the covariant Toeplitz operators
    in the sense of \cite{Del} are the analytic Berezin-Toeplitz operators
    defined by \eqref{eq:def_an_ber_toep}.
\end{enumerate}
\end{rmk}

As already mentioned, the main contribution of this work is in our proofs.
Recall first that the basic tools of analytic microlocal analysis, including analytic symbols and
stationary phase lemma, have been introduced a long time ago in \cite{BoKr},
\cite{SjAs} for the theory of analytic pseudo-differential operators. When we
try to apply these techniques to Berezin-Toeplitz operators, we face the
difficulty that the symbol product corresponding to operator composition is
only partially known, unlike the Moyal-Weyl product for the pseudodifferential
operators. A large literature exists on these products describing them
in terms of the K\"ahler metric, cf. \cite{MaMa} for instance, but the application to our problem is
not straightforward, as is attested by the attempt to prove that $\sum \hbar^{\ell}
\rho_{\ell}$ is analytic in \cite{HeLuXu}, or the complicated estimates of \cite{Del}.

So we have
to identify the characteristics of these products which allow the analytic
calculus. The convenient property we found is a particular growth of some
coefficients. 
To be more specific, define the {\em symbol} of a Berezin-Toeplitz operator $(S_k)$ as the series $\sum
\hbar^{\ell} \tau_{\ell}$ whose coefficients $\tau_{\ell} \in \Ci (M)$ are the restrictions
to the diagonal of the $\ti \tau_{\ell}$ in \eqref{eq:kernel_toep_intro}.
In \cite{oim}, these symbols were called non normalized covariant symbols,
because the Bezerin covariant symbol is obtained by normalising them.
Furthermore we proved  that  the
product of these symbols, corresponding to the operator composition, has the form $\sum \hbar^{\ell} f_{\ell} \star \sum
\hbar^m g_m = \sum \hbar^{\ell + m + p } A_p (f_{\ell}, g_{m} ) $ where
the $A_p$ are bidifferential operators with a local expression 
$$A_p ( f,g) = \sum_{|\al|,
  |\be| \leqslant p} a_{p, \al, \be} \frac{(\partial_{\bar{z}}^\be f)( \partial_{z}^\al
g) }{\al! \be !}. $$ The unit of $(\Ci (M)[[\hbar]], \star)$ is nothing else than $\sum
\hbar^{\ell} \rho_{\ell}$. 

The main estimate we will establish is that a holomorphic extension of the
$a_{p, \al, \be}$'s satisfies
\begin{gather} \label{eq:mest}
|\ti a_{p, \al , \be} | \leqslant C^{p+1} p! \end{gather} 
where $C$ does not depend on $p$, $\al$, $\be$; cf. Theorem \ref{th:main_estimate}. From this, we will deduce that
$\sum \hbar^\ell \rho_{\ell}$ is analytic, cf. Theorem
\ref{th:estimate_bergman}, and that the space of analytic symbols is closed
under $\star$, cf. Proposition \ref{prop:an_symb_alg}.

To do that, the two essential tools we will use are first the explicit
computation of the $A_{p}$ given in \cite{oim} and second a family of seminorms
already used in \cite{SjProp}. These seminorms are useful to
estimate coefficients of holomorphic differential operators and of their
compositions without using Leibniz formula. They were introduced in \cite{SjProp} to prove
that the inverse of an analytic symbol of a pseudodifferential operator is
analytic, a problem similar to ours. However there is the slight difference that
these seminorms are defined in \cite{SjProp} for the symbols themselves,
whereas we use them for the operators acting on symbols.

To prove the second part of Theorem \ref{th:intro}, that is the characterization of
analytic Berezin-Toeplitz operators in terms of their Schwartz kernel, a similar
difficulty arises: define the contravariant symbol of a Berezin-Toeplitz $(S_k)$
operator as the formal series $\sum \hbar^{\ell} f_{\ell}$ whose coefficients
are given in \eqref{eq:def_BT}; then  we have to show that the isomorphism $B$ of $\Ci ( M ) [[\hbar]]$ sending the contravariant symbol
 into the non normalised covariant symbol 
restricts to a bijection of the space of analytic symbols. To do this, we will
prove with the same methods
as before,  that $B$ and its inverse have the form $B = \sum \hbar^{\ell} B_{\ell}$ where the $B_{\ell}$ are
differential operators satisfying estimates similar to \eqref{eq:mest}. 

The paper is organised as follows. In section \ref{sec:symbolic-calculus}, we
recall the computation of the bidifferential operators $A_p$ from
\cite{oim} and prove our main estimates \eqref{eq:mest}. In section
\ref{sec:symb-bergm-kern}, we deduce that the symbol
of the Bergman projector is analytic. In section \ref{sec:analyt-berez-toepl},
we introduce analytic Berezin-Toeplitz operators by their Schwartz kernels,
and prove that they form an algebra with unit the Bergman kernel. In
section \ref{sec:multiplicator}, we prove that analytic Berezin-Toeplitz
operators can equivalently be defined by multipliers
\eqref{eq:def_an_ber_toep}. Last section is an appendix in two parts: in the
first part, we prove basic facts on analytic symbols which are essentially
known \cite{BoKr} and \cite{SjAs}; in the second part, we systematize some
of the techniques used in the previous parts. 

\vspace{.5cm}
\noindent
{\bf Acknowledgment:} I would like to thank Alix Deleporte for helpful
discussions about his work.

\section{Symbolic calculus} \label{sec:symbolic-calculus}

\subsection{Complexification} \label{sec:complexification}

Let $M$ be a complex manifold. We denote by $\con {M}$ the complex manifold,
which has the same underlying real manifold as $M$ but the opposite almost
complex structure. If $x\in M$, we denote by $\con x$ the corresponding point
of $\con M$.

The product $M \times \con {M}$ is a complexification of the
diagonal $\Delta_M = \{ (x, \con{x})/$ $x \in M\}$, in the sense that for any $x\in M$,
there exists a holomorphic chart $\psi: W \rightarrow \C^{2n}$ of $M \times \con M$ at $(x, \con
x)$ such that $\psi ( W \cap \Delta_M) = \psi (W)  \cap \R^{2n}$. This
has the consequence that any analytic function $f: \Delta_M \rightarrow \C$ has a
holomorphic extension $\tilde{f} : W \rightarrow \C$ on a neighborhood $W$ of
$\Delta_M$ in $M \times \con M$. We will often add a tilde to denote a
holomorphic extension. We will also often identify $\Delta_M$  with $M$.

Similarly, if $L \rightarrow M$ is a holomorphic line bundle, $\con{L}
\rightarrow \con M$ is the conjugate holomorphic line bundle, and any analytic section of the
restriction of $L
\boxtimes \con L$ to the diagonal has a holomorphic extension to a
neighborhood of the diagonal. Consider a Hermitian metric 
of $L$, which is analytic in the sense that for any local
holomorphic frame $s: U \rightarrow L$, $x \rightarrow |s(x)|^2$ is a analytic function of $U$.
Then the section of $L \boxtimes \con {L} \rightarrow \Delta_M$, sending
$(x, \con x)$ to $|u|^{-2}  u \otimes \con u$, $u \in L_x$ is
analytic, so it has a holomorphic extension $E : W \rightarrow L \boxtimes
\con {L}$. If $s$ is a holomorphic frame as above and $|s(x)|^2 = \exp ( -
\varphi (x))$, then we have the local expression for $E$
$$ E(x, \con{y} ) = e^{ \tilde{\varphi} (x, \con{y}) } s(x) \otimes
\con{s(y)} $$
where $\tilde \varphi$ is a holomorphic extension of $\varphi$. 
Sometimes, it is convenient to identify the tensor product $L_x \otimes \con
L_x$ with $\C$ through the metric, so that the restriction of
$L \boxtimes \con L$ to the diagonal becomes the trivial line bundle of $\Delta_M
$. With this convention, the restriction of $E$ to the diagonal is
simply the constant function equal to $1$. 

\subsection{Product formula} \label{sec:product-formula}
Consider a compact complex manifold $M$, with two Hermitian holomorphic line
bundles $L$ and $L'$. We assume that $L$ is positive. Let $f$ and $g$ be two analytic
functions of $M$.

Introduce as in Section \ref{sec:complexification}, holomorphic extensions $E : W \rightarrow L
\boxtimes \con L$, $E' : W \rightarrow L' \boxtimes \con L'$, $\tilde{f} : W
\rightarrow \C$, $\tilde{g} : W \rightarrow \C$ defined on the same
neighborhood $W$ of $\Delta_M$.  
The fact that $L$ is positive has the consequence that $|E(x,y_c)| < 1$  when
$ y_c \neq \con x$ and $(x,y_c)$ is sufficiently close to the diagonal \cite[Proposition 1]{oim}.
So
restricting $W$ if necessary, we can assume that $|E|<1$ on $W \setminus
\Delta_M$. For any integer $k \geqslant 1$ and $(x,y_c) \in W $, set
\begin{gather} \label{eq:defP_k}
  E_k (x,y_c) := \biggl( \frac{k}{2\pi} \biggr)^n  E^k (x,y_c) \otimes E' (x, y_c)
\end{gather}
Choose a smooth compactly supported function $\rho: W \rightarrow
\R$ which is equal to $1$ on a neighborhood of $\Delta_M$. Define
\begin{gather*}  T_k (x,y_c) =   \rho (x,y_c) E_k(x,y_c)  \tilde{f} (x,y_c) ,
 \\ S_k(x,y_c) =   \rho (x,y_c) E_k(x,y_c)  \tilde{g} (x,y_c). 
\end{gather*}
So $T_k$ and $S_k$ are smooth sections of $(L^k \otimes L') \boxtimes (\con{ L^k
  \otimes L'})$. They are Schwartz kernels of operators $T_k$, $S_k$ acting of $\Ci ( M ,
L^k \otimes L')$. Our convention for operator kernels is 
\begin{gather} \label{eq:S_Kernel}
(T_k \psi )(x) = \int_M T_k(x,\con y) \cdot \psi (y) \; d \mu (y), \qquad \psi
\in \Ci ( M , L^k \otimes L')
\end{gather}
where $\mu$ is the Liouville measure of $M$ and the the dot stands for the
scalar product of $(L^k \otimes L')_y$. Let $r_k$ be  the restriction of
the Schwartz kernel of $S_k \circ T_k$ to the diagonal
$$ r_k (x) = \int_{M} S_k (x,\con{y} ) \cdot T_k (y, \con{x} ) \; d\mu (y)
$$
By \cite{oim}, the sequence $(r_k)$ has an asymptotic expansion
$$ r_k (x) = \biggl(\frac{k}{2\pi} \biggr)^n \sum_{\ell=0}^N k^{-\ell} A_\ell (f,g) (x) +
\bigo ( k^{n- N-1}) , \qquad \forall N$$
and we can compute explicitly the coefficients $A_\ell
(f,g)(x)$ as follows.

For any $x_0 \in M$, choose holomorphic frames $s$, $s'$ of $L$ and $L'$ over
a connected  open set $U \ni x_0$. Set $|s|^2 = \exp ( - \varphi)$ and $|s'|^2 = \exp ( - \varphi')$.
Restricting $U$ if necessary, we have holomorphic extensions $\tilde \varphi$,
$\tilde \varphi'$ on $U \times \con{U} \subset W$. So $|E_k |^2 =  (k/2\pi)^{2n} \exp ( -k \psi - \psi')$
where
\begin{gather*}
  \psi (x,y) = - \tilde{\varphi} ( x, \con y ) -  \tilde{\varphi} ( y, \con x
  )  +  \tilde{\varphi} ( x, \con x )  +  \tilde{\varphi} ( y, \con y ) \\
   \psi' (x,y) = - \tilde{\varphi} '( x, \con y ) -  \tilde{\varphi}' ( y, \con x
)  +  \tilde{\varphi} '( x, \con x )  +  \tilde{\varphi}' ( y, \con y )
\end{gather*}
With the identification $L_x \otimes \con{L}_x \simeq \C$ given by the metric,
$E_k(x,\con y) \cdot E_k ( y, \con x) = |E_k (x,y) |^2$. So
$$ r_k (x) = \biggl( \frac{k}{2 \pi} \biggr)^{2n} \int_M  e^{ - k
  \psi (x,y) -\psi' (x,y)} \tilde{f}(x, \con y ) \tilde{g} ( y , \con x )  \rho' (x,y) \delta (y) \; d
\mu_L (y). $$
where $\rho'(x,y) = \rho ( x, \con y ) \rho ( y, \con x)$, $\mu_L$ is the Lebesgue measure of $\C^n$ and $\mu (x) = \delta (x) \mu_L(x)$.
Choosing a holomorphic chart with domain $U$, we consider $U$ as an open set
of $\C^n$. Let
$G_{ij} = \partial^2 \varphi/ \partial z_i \partial \con z_j$ and recall that
for any $x \in U$, 
$(G_{ij}(x))$ is positive definite because $L$ is a positive line bundle.

\begin{thm} \label{th:local_expression}
  For any $x \in U$, we have 
\begin{gather} \label{eq:local_expression}
  A_{\ell} (f,g) (x)  = b (x)\sum_{m =0 }^{2\ell} \frac{ \Delta^{\ell +m} (c^m
  d c' ) }{m ! ( \ell + m ) !} (x,0)
\end{gather}
where for $u \in \C^n$ sufficiently small 
\begin{xalignat*}{2}
  & b(x)  = (\det (G_{ij} (x)) )^{-1}, \quad c(x,u) =  \textstyle{\sum}_{i,j = 1}^n G_{ij} (x) u_i  \con{u}_j - \psi (x ,  x + u)  \\
  & d(x,u)  = \tilde{f} (x, \con{x} + \con{u} ) \tilde{g} ( x + u , \con {x} ),
  \qquad   c'(x,u)  = e^{-\psi' (x,x+u)} \delta (x+ u)
\end{xalignat*}
and  $\Delta  = \textstyle{\sum}_{i,j=1}^n G^{ij} (x) \partial_{u_i}
  \partial_{\con{u}_j}$, $(G^{ij}(x))$  being the inverse of
  $(G_{ij}(x))$.
\end{thm}

\begin{proof}We can rewrite the integral in terms of the functions  $c$, $c'$, 
  $d$
  $$ r_k (x) = \biggl( \frac{k}{2 \pi} \biggr)^{2n} \int  e^{ -
    k \sum G_{ij} (x) u_i \con{u}_j} e^{k c(x,u)} d (x,u) c'(x,u)  \rho'
  (x,x+u)  d\mu (u) $$
Notice that $c(x,u) = \bigo ( |u|^3)$. The result follows from Laplace's method, cf.
for instance \cite[Theorem 7.7.5]{Horm}. 
\end{proof}
The proof in \cite{oim} was much longer because we computed there the
Schwartz kernel of $S_k \circ T_k$ on $M \times \con{M}$ up to a $\bigo (
k^{-\infty})$, which requires the stationary phase lemma  with complex valued
phase depending on parameters. In Section \ref{sec:analyt-berez-toepl}, we will see the similar result in the analytic setting.

Since $c' (x,0) = \delta (x) = \det (G_{ij} (x))$, we have
$$A_0 ( f,g)= fg.$$
In the sequel, we will use the following holomorphic extension of
\eqref{eq:local_expression}. For any holomorphic function $(u,v ) \rightarrow \ti d(u,v)$ defined on a
neighborhood of the origin in $\C^n
\times \C^n$, we set
\begin{gather} \label{eq:def_Pell}
P_{\ell} ( \ti d) (x,y_c) = \tilde b (x, y_c)\sum_{m =0 }^{2\ell} \frac{ \wt
  \Delta^{\ell +m} (\ti c^m
  \ti d \ti c' ) }{m ! ( \ell + m ) !} (x,y_c, 0, 0 ) 
\end{gather}
Here $\ti b$, $\ti c$, $\ti c'$ are holomorphic extensions of $b$, $c$, $c'$
respectively. So $\ti b$ depends holomorphically on $x$, $y_c$ and $b( x) = \ti b (x, \con x)$, $\ti c$ and $\ti c'$ are holomorphic functions of the variables $(x,y_c, u,v)$ such that $c
( x,u) = \ti c ( x,\con x , u , \con u )$, $c' ( x,u) = \ti c' ( x,\con x , u
, \con u )$. Similarly, $\wt \Delta $ is the operator $ \sum_{i,j=1}^n \wt G^{ij} (x,y_c)
\partial _{u_i} \partial_{v_j}$. Then the function $A_{\ell} (f,g)$ of Theorem \ref{th:local_expression} has
the holomorphic extension
\begin{gather} \label{eq:B_extension}
\tilde{A}_{\ell} (\tilde f ,  \tilde g) (x, y_c) = P_{\ell} ( \ti d_{(x,y_c)})
(x,y_c) 
\end{gather}
with $\ti d_{(x,y_c)} (u,v)  = \tilde f (x, y_c + v) \tilde g ( x+ u, y_c)$.

Even if we don't need it, let us observe that everything can be explicitly computed
in terms of $\tilde \varphi$ and $\tilde \varphi'$. Indeed, $( \wt G ^{ij} (x,y_c))$
is the inverse of $\wt G_{ij} ( x,y_c) = (\partial^2 \tilde \varphi / \partial
x_i \partial y_{c,j} ) ( x, y_c) $, $\ti b = \det \wt G_{ij}$, $\ti c (x,y_c, u, v)  =  \tilde{\varphi} ( x, y_c + v) +
\tilde{\varphi} ( x + u , y_c) - \tilde{\varphi} (x,y_c)  - \tilde{\varphi}
(x+u, y_c + v)  + \sum_{i,j} \wt G_{ij} (x, y_c) u_i v_{j}$ and there is a similar
formula for $\ti c' (x,y_c, u, v)$. 

\subsection{Main estimates} \label{sec:main-estimates}

Consider the local expression of $A_{\ell} (f,g)$ given in Theorem
\ref{th:local_expression}.
\begin{lm} \label{lem:maj_ordre_derivee}
  We have
$$\frac{1}{\ell !}  A_{\ell} (f,g) = \sum_{|\al|, |\be| \leqslant
  \ell}    a_{\ell, \al , \be} \frac{  (\partial_{\con z}^\be f)  
(\partial_z ^\al g) }{\al! \be !}  $$
where $\al$, $\be$ are multi-indices of $\N^n$, $a_{\ell, \al, \be}$ are
analytic function of $U$, the derivatives are $\partial_z ^\al = \partial_{z_1}^{\al(1)} \ldots \partial_{z_n}^{\alpha (n)}$, $\partial_{\bar z}^\be = \partial_{\bar z_1}^{\be(1)} \ldots \partial_{\bar z_n}^{\be (n)}$.
\end{lm}
\begin{proof}
Observe first that  $(\partial^\al_u \partial^\be_{\bar{u}} d)(x,0) =
(\partial_{\bar z}^\be f) (x)
(\partial_z ^\al g) (x)$. Second, we have $(\partial^\al_u c)(x,0) =
(\partial^\be_{\bar{u}} c)(x,0) =0$, so in the Taylor expansion of $u
\rightarrow c(x,u)$ at $0$, only appear monomials of the form $u^\al
\bar{u}^\be$ with $|\al|, |\be| \geqslant 1$. Thus the Taylor expansion of $u
\rightarrow c^m (x,u)$ has only monomials of the form $u^\al
\bar{u}^\be$ with $|\al|, |\be| \geqslant m$. Now $\Delta^{m+\ell}$ is a
linear combination of $\partial_u^{\al} \partial_{\bar u}^\be$ with $|\al|=
|\be| = m+ \ell$. By the previous remark, expanding $\Delta^{m+ \ell} ( c^m
d c') (x,0)$ with the Leibniz rule, only the terms with at least $m$ holomorphic
derivatives and $m$ antiholomorphic derivatives on $c^{m}$
will not be zero, so it remains at most $\ell$ holomorphic and $\ell$
antiholomorphic derivatives on $d$.   
\end{proof}

The functions $a_{\ell, \al, \be}$ can be computed by expanding (\ref{eq:local_expression}). Since the functions $b$, $c$, $c'$ in (\ref{eq:local_expression}) have holomorphic extensions to $(x,x_c) \in U \times \con U$, each $a_{\ell, \al, \be}$ has a holomorphic extension $\tilde a _{\ell, \al, \be} : U \times \con U \rightarrow \C$. 
We can now state our main estimate. 
\begin{thm} \label{th:main_estimate}
  For any compact subset $K$ of $U \times \con U$, there exists $C>0$ such that for any $\ell$, $\al$, $\be$,
\begin{gather} \label{eq:main_ineq}
| \tilde{a}_{\ell, \al , \be} (x,y_c) | \leqslant C^{\ell+1}, \qquad
\forall (x,y_c) \in K .
\end{gather}
\end{thm}

It is possible to deduce Theorem \ref{th:main_estimate} directly from Theorem \ref{th:local_expression}
by repeated applications of Leibniz formula. Instead, we will present a
less computational argument based on a family of seminorms considered in \cite{SjProp}. 
Let $\Om_r$ be the open ball of $\C^n \times \C^n$ centered at the
origin with radius $r$. Let $\mathcal{B}(\Om_r)$ be the space of holomorphic
bounded functions of $\Om_r$, with the norm $\| f \|_r = \sup_{\Om_r} |f|$.
For any bounded operator $P:\mathcal{B} ( \Om_t)
\rightarrow \mathcal{B} (\Om_s)$, we denote by  $ \| P \|_{t,s} = \sup \{ \|
Pf \|_s /\; \| f \|_t \leqslant 1 \}$ the corresponding norm. 

\begin{lm} \label{lem:dertech}$  $
  \begin{enumerate}
  \item   There exists $C>0$ such that for any $\ga \in \N^{2n}$ and $0<s<t$
    we have $$  \bigl\| \partial^\ga \bigr\|_{t,s}
  \leqslant \frac{C^{| \ga | + 1 } \ga ! }{ ( t-s )^{|\ga|}}.$$
  \item  Let $t_0 >0$, $p, q \in \N$, $C>0$ and $P,Q$ be two operators $ \mathcal B (\Om_{t_0}) \rightarrow \mathcal{B} (
    \Om_{t_0})$ such that for any $0<s <t \leqslant t_0$,   $\| P \|_{t,s} \leqslant
(Cp)^p/(t-s)^p$ and $\| Q \|_{t,s} \leqslant
(Cq)^q/(t-s)^q$. Then for any $0<s <t \leqslant t_0$,  $$ \| P \circ Q \|_{t,s}
\leqslant \frac{(C(p+q))^{p+q}}{(t-s)^{p+q}}. $$
  \end{enumerate}
\end{lm} 

\begin{proof}
First assertion follows from Cauchy inequality. For the second
assertion, we follow \cite[page 69]{SjProp}: let 
 $r \in [s,t]$ be such that $r-s = \frac{p}{p+q} (t-s)$ and
$t-r=\frac{q}{p+q} ( t -s )$. Then
$$ \| P \circ Q \|_{t,s}  \leqslant  \| P \|_{r,s} \| Q \|_{t,r} \leqslant
\frac{(Cp)^p}{(r-s)^p} \frac{(Cq)^q}{(t-s)^q} =
\frac{(C(p+q))^{p+q}}{(t-s)^{p+q}} $$
as was to be proved.
\end{proof}
\begin{proof}[Proof of Theorem \ref{th:main_estimate}]
  The $\tilde a_{\ell, \al, \be}$ are given in terms  of the
  operator $P_\ell$ defined in \eqref{eq:def_Pell} by
\begin{gather} \label{eq:7} 
  \tilde a _{\ell, \al, \be} (x,y_c) = P_{\ell} ( u^\be v^{\al} ) (x,y_c)
\end{gather}
 For any $0<r
\leqslant 1$, $\| u^\be v ^\al \|_{r} \leqslant 1$. We will estimate the norm of $\mathcal{B} ( \Om_r) \rightarrow \C$, $\ti d  \rightarrow
P_{\ell} (\ti d) (x,y_c)$. In the sequel, the constants $C$, $C'$ depend on
$(x,y_c)$, but remain bounded as $(x,y_c)$ stays in a compact set. 

Since the function $\ti c ( x,y_c,u,v)$ vanishes when $u = v=0$, the multiplication operator
\begin{gather} \label{eq:5}
\mathcal{B} ( \Om_r)
  \rightarrow \mathcal{B} ( \Om_r) , \qquad \tilde d  (u,v) \rightarrow \ti c(x,y_c,
  u,v) \ti d (u,v)
\end{gather}
  has a norm  smaller than $1$ when $r$ is sufficiently
  small. So the same holds for the multiplication by $\ti c^m$.

  If $r$ is sufficiently small, there exists a constant $C$ such that for any $m$,
\begin{gather} \label{eq:6}
  \mathcal{B} ( \Om_r)
  \rightarrow \C, \qquad  \ti d (u,v) \rightarrow ( \wt \Delta^m \ti d) (0,0)
\end{gather}
  has a norm smaller than $C^m m^{2m}$. Indeed, by assertion 1 of Lemma \ref{lem:dertech}, $ \| \wt \Delta \|_{t,s}
  \leqslant C/ ( t-s )^{2}$ when $0<s<t$ are sufficiently small. So by
  assertion 2 of Lemma \ref{lem:dertech}, for any $m \in \N$, $ \| \wt \Delta ^m  \|_{t,s}
  \leqslant C^m m^{2m}/ ( t-s )^{2m}$ with the same constant $C$. To conclude,
  choose
  $(t,s ) = ( r,0)$ and replace $(C/r^2)$ by $C$. 

The norm estimates of (\ref{eq:5}) and (\ref{eq:6})  imply through (\ref{eq:7}) that
$$ | \tilde a _{\ell, \al , \be}  (x,x_c)| \leqslant C' \sum_{m = 0 }^{2\ell}
C^{\ell + m }  \frac{ (\ell + m ) ^{2 ( \ell  + m ) }}{(\ell + m )! m !}
$$
where $C'$ is the product of upper bounds of $|\ti b |$ and $|\ti c' |$. Using that
$p^p \leqslant e^p p!$ and  
$ (p+q ) ! \leqslant 2^{p+q} p! q!$, we get that
$$   \frac{ (\ell + m ) ^{2 ( \ell  + m ) }}{(\ell + m )! m !} \leqslant
\biggl( \frac{e}{2}\biggr)^{2 ( \ell + m )} \frac{ ( 2 ( \ell + m )) !}{( \ell + m ) ! m !}
\leqslant e^{2 ( \ell  + m ) } \frac{ ( \ell + m ) !}{ m !} \leqslant (2
e^2)^{\ell + m }\ell ! $$
so $| \tilde a _{\ell, \al , \be}  (x,x_c)| \leqslant C' (2\ell +1) (2 Ce^2)^{3
  \ell } \ell !$ and we conclude easily.
\end{proof}

\section{Symbol of the Bergman kernel} \label{sec:symb-bergm-kern}

The Bergman projector $\Pi_k$ of $L^k \otimes L'$ is the projector of $\Ci (M , L^k \otimes L' )$ onto the subspace $\Hilb_k = H^0 ( M, L^k \otimes L')$ consisting of holomorphic sections. The Bergman kernel $(x,y) \rightarrow \Pi_k(x,y)$ is the Schwartz kernel of $\Pi_k$, so
$$ \Pi_k
(x,y_c) = \sum_{i=1}^{d_k} \Psi_i (x) \otimes \con{\Psi}_i (y_c).$$
where $(\Psi_{i})_{i=1, \ldots, d_k} $ is any orthonormal basis of $H^0 ( M , L^k \otimes L')$.

It was proved in \cite{Ze} that the restriction to the diagonal of the Bergman kernel has an asymptotic expansion
\begin{gather} \label{eq:coef_bergman} 
  \Pi_k (x, \con x) = \biggl( \frac{k}{2 \pi} \biggr)^n \sum_{\ell =0 }^N k^{-\ell} \rho_\ell (x) + \bigo (k^{ n- N-1}), \qquad \forall N
\end{gather}
with smooth coefficients $\rho_{\ell} \in \Ci (M)$. 

\begin{thm} \label{th:estimate_bergman}
There exist an open neighborhood $W$ of $\Delta_M$ and a constant $C>0$ such that for
  any $\ell$, the function $\rho_\ell$ has a holomorphic extension $\wt
  \rho_{\ell}$ to $W$ satisfying
$$ | \wt {\rho}_{\ell} (x,y_c ) | \leqslant C^{\ell+1} \ell^\ell , \qquad
\forall (x,y_c) \in W.$$ 
\end{thm}

As already mentioned in the introduction, this result was proved in {\cite{RoSjVu}}.
For the proof we will compute the functions  $\rho_{\ell}$ from the bidifferential operators $ A_{\ell}$ considered previously, by using that $(\Pi_k)$ is the unit of an algebra of Berezin-Toeplitz operators.

These operators will be defined as families $(T_k \in \op{End}( \Hilb_k), \; k \in \N)$  whose Schwartz kernel has a particular form. Here the Schwartz kernel is given by  
\begin{gather} \label{eq:schwartz_k}
T_k (x,y_c) = \sum_{i=1}^{ d_k} (T_k \Psi_i )(x) \otimes \con{\Psi}_i (y_c) , \qquad (x,y_c ) \in M \times \con{M} .
\end{gather}
where $(\Psi_i)$ is an orthonormal basis of $\Hilb_k$ as above. 
Conversely, we recover $T_k\Psi $ from its Schwartz kernel with the integral (\ref{eq:S_Kernel}).
\footnote{When $\Psi$ is smooth,  the right-hand side of (\ref{eq:S_Kernel})
  is equal to $T_k \Pi_k \Psi$. So we have implicitly identified the
  endomorphisms of $ \Hilb_k$ with the operators $T_k$ acting on $\Ci (M, L^k
  \otimes L')$ and satisfying  $\Pi_k T_k \Pi_k =T_k$.}  

Following \cite{oim}, 
we define a {\em Berezin-Toeplitz operator} as any family $( T_k)  \in \prod_k \op{End} \Hilb_k$  with Schwartz kernels of the form 
\begin{gather} \label{eq:1}
  T_k (x,y) = E_k (x,y) \sum_{\ell =0 }^N k^{-\ell}  \wt{\tau}_\ell (x,y) + \bigo (k^{ n- N-1}), \qquad \forall N
\end{gather}
where the sections $E_k$ are defined in (\ref{eq:defP_k}) on $W \subset M \times \con M$ and for any $\ell$, $\wt \tau_{\ell}$ is a smooth function of $M \times
\con M$ supported in $W$ such that $\con \partial \wt \tau_{\ell}$ vanishes to infinite
order along the diagonal. 
Denote by $\tau_{\ell}$ the restriction of $\wt \tau_{\ell}$ to the diagonal. We call the formal series $\sum \hbar^\ell \tau_\ell$ the symbol of $(T_k)$. Let $\mathcal{T}$ be the space of Berezin-Toeplitz operators and define the application
$$ \sigma:  \mathcal{T} \rightarrow \Ci(M)[[\hbar]], \qquad (T_k) \rightarrow \sum_{\ell=0}^\infty \hbar^{\ell} \tau_{\ell}.$$ 
We deduced in \cite{oim} from \cite{BoSj} that the Bergman kernel itself has the form (\ref{eq:1}). In other words, $(\op{id}_{\Hilb_k})$ belongs to $\mathcal{T}$. By (\ref{eq:coef_bergman}), its symbol if $\sum \hbar^\ell \rho_{\ell}$. 
We also proved in \cite{oim} that  $\mathcal{T}$ is closed under product and the map  $\si$ is onto with kernel the space of families $(T_k \in \op{End} \Hilb_k)$ such that $\| T_k \| = \bigo ( k^{-\infty})$. So $\mathcal{T}$ is a subalgebra of $\prod_k \Hilb_k$ and $\bigo ( k^{-\infty})$ being an ideal of this subalgebra, $\Ci (M) [[\hbar]]$ inherits a product $\star$. This product has been actually computed in Theorem \ref{th:local_expression}
\begin{gather} \label{eq:prod_star}
\sum_m \hbar^m f_m \star \sum_p \hbar^{p} g_{p} = \sum_{\ell,m,p} \hbar^{\ell+ m + p } A_\ell (f_m, g_p). 
\end{gather}
So $\sum \hbar ^{\ell} \rho_{\ell}$ is the unit of the associative algebra $(\Ci ( M) [[\hbar]], \star)$. 

By Theorem \ref{th:local_expression}, the $A_{\ell}$ have analytic
coefficients, so $\mathcal{C}^{\omega} (M)[[\hbar]]$ is a subalgebra of $(\Ci (M)[[\hbar]], \star)$, in particular the $\rho_{\ell}$'s are analytic. For an operator whose symbol has analytic coefficients, we can simplify slightly the formula (\ref{eq:1}) by choosing for $\wt{\tau}_\ell$ a function holomorphic on a neighborhood of the diagonal. This has the advantage that the right hand side of (\ref{eq:1}) is uniquely determined on a neighborhood of the diagonal by its restriction to the diagonal. Furthermore, these operators form a subalgebra of $\mathcal{T}$, so we could have considered only them. But in the theory developed in \cite{oim}, this was not possible because we worked with a smooth metric not necessarily analytic for the bundle $L$.  

\begin{lm}
  We have for each $m \geqslant 1$,
\begin{gather} \label{eq:2}
  \rho_{m} = \sum _{\substack{\ell \geqslant 1, \; (i_1, \ldots, i_{\ell}) \in \Z_{>0}^\ell,\\ i_1 + \ldots + i_\ell = m} } Q_{i_1} \ldots Q_{i_{\ell}} (1)
\end{gather}
  where for any $i \in \Z_{>0}$, $Q_i $ is the differential operator $Q_i (f) = - A_i (f,1)$.
\end{lm}
\begin{proof}
  Since $\sum \hbar^\ell \rho_\ell$ is the unit of $\star$, we have $\sum \hbar^{\ell+m} A_\ell ( \rho_m , 1) = 1$. Using that $A_0 ( f,g) =fg$, we obtain
  $$ \rho_0 =1 , \qquad \rho_1 =  Q_1 ( \rho_0), \quad \rho_2 = Q_2 ( \rho_0) + Q_1 ( \rho_1)  $$
  and more generally $ \rho_{m} = Q_m (\rho_0) + Q_{m-1} ( \rho_1) + \ldots + Q_1 ( \rho_{m-1})$. This can be solved inductively by
  \begin{xalignat*}{2} 
     \rho_0 & = 1 , \qquad \rho_1 = Q_1 ( 1), \\ \rho_2 &  = ( Q_2(1) + Q_1(Q_1 (1))= (Q_2 + Q_1^2) ( 1) \\
    \rho_3 & = Q_3 (1) + Q_2 (Q_1(1)) + Q_1 ((Q_2 + Q_1^2) ( 1)) \\ & = ( Q_3 + Q_2Q_1+ Q_1Q_2 + Q_1^3 ) (1) 
  \end{xalignat*}
  and more generally we obtain (\ref{eq:2}).
\end{proof}

\begin{proof}[Proof of Theorem \ref{th:estimate_bergman}]
  Locally, we can define the holomorphic extension $\wt Q _\ell ( \tilde f ) = \wt A_{\ell} ( \tilde f,
1)$ of $Q_\ell (f)$ with $\wt A$ given in \eqref{eq:B_extension}. If in the right-hand side of Equation \eqref{eq:2}, we replace each
$Q_{i_j}$ by $\wt{Q}_{i_j}$, we obtain a holomorphic extension of $\rho_m$. 
We have 
$$  \wt{Q}_{\ell} ( \tilde{f} ) (x,y_c) = - \ell ! \sum_{|\al| \leqslant \ell }
\frac{1}{\al!}\tilde{a}_{\ell, \al, 0} (x,y_c) 
\partial_{y_c} ^\al \tilde f  (x,y_c). $$
Assume the coordinates are centered at $x_0$ and use them to identify a neighborhood of
$(x_0, \con x_0)$ with a neighborhood of the origin in $\C^{n} \times \C^n$.
Introduce the same seminorms $\| \cdot \|_{t,s}$ as before Lemma \ref{lem:dertech} . We have for
any $\ell \geqslant 1$ and $0<s<t$ sufficiently small
\begin{gather} \label{eq:step1}
  \| \wt Q_{\ell} \|_{t,s} \leqslant \frac{ (C'\ell)^{\ell}  }{(t-s)^\ell}
\end{gather}
for some $C'>0$. Indeed choose $C$ so that first part of Lemma
\ref{lem:dertech} holds and Theorem \ref{th:main_estimate} as well. So we have
  $$\frac{1}{\al!} \bigl\| \tilde{a}_{\ell, \al, 0} \partial_{y_c}^\al
  \bigr\|_{t,s} \leqslant \frac{C^{\ell+ | \al | + 1 }}{ ( t-s )^{|\al|}} .$$
  Assuming that $C\geqslant 1$ and $|\al| \leqslant \ell$, we have $(C/(t-s))^{|\al|}
  \leqslant ( C/(t-s))^{\ell}$. Furthermore the number of multiindices $\al
  \in \N^n$ with $| \al| \leqslant \ell$ is ${ \ell + n -1 \choose n-1 }$
  which is $\leqslant 2^{\ell + n -1}$. So the number of $\al$ with $|\al| \leqslant \ell$ is smaller that $2^{\ell + n }$, so 
  $$ \frac{1}{\ell!} \| \wt Q_{\ell} \|_{t,s} \leqslant \sum_{|\al | \leqslant \ell }
    \frac{C^{\ell+ | \al | + 1 }}{ ( t-s )^{|\al|}} \leqslant  2^{\ell + n }
    \frac{C^{2\ell+   1 }}{ ( t-s )^{\ell}}$$
  We obtain \eqref{eq:step1}  by using that $\ell ! \leqslant \ell^\ell$ and setting $C' = 2^{n+1}
  C^3$.

  Now \eqref{eq:step1} implies by assertion 2 of Lemma \ref{lem:dertech} that for any $(i_1, \ldots , i_{\ell}
) \in \Z_{>0}^\ell$
\begin{gather} \label{eq:unedeplus}
\| \wt Q_{i_1} \circ \ldots \circ \wt Q_{i_\ell} \|_{t,s} \leqslant \frac{(C'
  m )^m}{(t-s)^m}
\end{gather}
where $m = i_1+ \ldots  + i_\ell$.  
Recall that 
$$  \tilde \rho_{m} = \sum _{\substack{\ell \geqslant 1, \; (i_1, \ldots,
    i_{\ell}) \in \Z_{>0}^\ell,\\ i_1 + \ldots + i_\ell = m} } \wt Q_{i_1}
\ldots \wt Q_{i_{\ell}} (1). $$
Since the number of terms in the sum is $2^{m-1}$, we deduce from
\eqref{eq:unedeplus} with $t =2s$ and $s$ sufficiently small that
$$ \| \tilde \rho_m \|_s \leqslant 2^{m-1}  \frac{(C' m )^m}{s^m} ,$$
which concludes the proof.
\end{proof}

\section{Analytic Berezin-Toeplitz operators} \label{sec:analyt-berez-toepl}

\subsection{Analytic symbols}

An {\em analytic symbol} of $M$ is a formal series $\sum_{\ell=0}^\infty  \hbar^{\ell} f_\ell$ with coefficients in $\mathcal{C}^{\omega} (M)$ such that there exist a neighborhood $W$ of $\Delta_M$ in $M \times \con {M}$ and a constant $C>0$ so that each $f_{\ell}$ has a holomorphic extension $\tilde{f}_{\ell}$ to $W$ satisfying
\begin{gather} \label{eq:ineq_symb_an}
  |\tilde{f}_{\ell} (x) | \leqslant C^{\ell +1 } \ell !, \qquad \forall x \in W.
\end{gather}
Let $\mathcal{S}^{\omega}$ be the subspace of $\Ci (M) [[\hbar]]$ consisting of analytic symbols.
Recall the product $\star$ of $\Ci (M)[[\hbar]]$ introduced in (\ref{eq:prod_star}).

\begin{prop} \label{prop:an_symb_alg}
$\mathcal{S}^{\omega}$ is a subalgebra of $(\Ci (M)[[\hbar]], \star )$.
\end{prop}

\begin{proof}
  By Theorem \ref{th:estimate_bergman}, the unit $\sum \hbar^\ell \rho_{\ell}$ of $\star$ is an analytic symbol. Let us prove that the product of two analytic symbols $\sum \hbar^\ell g_{\ell}$ and $\sum \hbar^\ell h_{\ell}$ is analytic. Recall the holomorphic extension $\wt A_{\ell}$ on $\wt U = U \times \con U \ni (x_0, \con x_0)$ defined in (\ref{eq:B_extension}). Restricting $U$ if necessary, each $f_{\ell}$, $g_{\ell}$ has an holomorphic extension to $\wt U$ and    
  we have to prove that
  $$\tilde{f}_{\ell} = \sum_{m+ p + q = \ell} \wt {A}_{m} (\tilde g_p, \tilde h_q )$$ satisfies (\ref{eq:ineq_symb_an}) on an open neighborhood $\wt V$ of $(x_0, \con{x}_0)$. This follows from the second part of Lemma \ref{lem:ledernier} and the fact that if $\wt V $ has a compact closure in $\wt U $, then $\wt A_{\ell}$ is continuous  $\mathcal{B} ( \wt U) \times  \mathcal{B} ( \wt U) \rightarrow \mathcal{B} ( \wt V )$ with a norm smaller than $ (C')^{\ell+1} \ell !$.
  Here $\mathcal{B} ( \wt V)$ is the space of bounded holomorphic function of $\wt V$ with the sup norm $| \cdot |_{\wt V}$.

 To prove the continuity and the norm estimate of $\wt A_{\ell}$, choose $C$  as in Theorem \ref{th:main_estimate}. Replacing $C$ by a larger constant, we have by Cauchy's inequality
  $ \frac{1}{\ga!} |\partial^{\ga} \tilde f |_{\wt V} \leqslant C^{|\ga|} |f|_{\wt U}$. So
  $$   \Bigl| \tilde{a}_{\ell, \al , \be} \frac{(\partial_{y_c}^{\be} \tilde{f})
    (\partial_{x}^{\al} \tilde{g})}{\al ! \be !} \Bigr|_{\wt V} \leqslant C^{3 \ell + 1} |\tilde f |_{\wt U} |\tilde g |_{\wt U}
  $$
for $|\al|$, $|\be| \leqslant \ell$.
The number of $\al$, $\be$ satisfying this condition being $\leqslant 4^{\ell +n -1}$, we obtain
$$ \bigl| \wt {A}_{\ell} ( \tilde f, \tilde g ) \bigr|_{\wt V} \leqslant \ell ! 4^{\ell +n -1} C^{3 \ell + 1} |\tilde f |_{\wt U} |\tilde g |_{\wt U} $$
which proves the claim with $C' = \max ( 4C^3 , 4^{n-1} C)$. 
  \end{proof}

\subsection{Berezin-Toeplitz operators}

\begin{dfn} \label{def:bertoepan}
An {\em analytic Berezin-Toeplitz operator} is a family $(T_k) \in  \prod_k \op{End}(\Hilb_k)$ whose Schwartz kernels satisfy
\begin{enumerate}
\item[{\em i.}] \label{item:dec_exp}
  for any compact subset $K$ of $M \times \con{M}$ not intersecting $\Delta_M$, there exists $C_K>0$ such that $|T_k (x,y_c) | \leqslant e^{-k/C_K} $ on $K$.
\item[{\em ii}.]
  on a neighborhood $W$ of the diagonal
\begin{gather} \label{eq:def_BT_analytic}
  T_k (x,y_c) = E_k (x,y_c) \sum_{\ell =0 } ^{ \lfloor  \ep k \rfloor
} k^{-\ell}  \tilde{f}_{\ell} (x,y_c)  + \bigo ( e^{-k/C'})
\end{gather}
  where the sections $E_k$ are defined in  (\ref{eq:defP_k}),  $\sum \hbar^{\ell} f_{\ell}$ is an analytic symbol, each
  $\tilde{f}_{\ell}$ is a holomorphic extension of $f_{\ell}$ to $W$
  satisfying (\ref{eq:ineq_symb_an}) for a constant $C$, $0 < \epsilon < 1/C$, $C'>0$  and the $\bigo$ is
  uniform on $W$. \qed 
\end{enumerate}
 \end{dfn}

Comparing to the introduction, the definition here is the characterization in theorem \ref{th:intro}. The fact that this definition is equivalent to the expansion (\ref{eq:def_an_ber_toep}) will be proved in Section \ref{sec:multiplicator}. 

Let $\mathcal{T}^\omega$ be the space of analytic Berezin-Toeplitz operators and $\si: \mathcal{T}^\omega \rightarrow \mathcal{S}^\omega$ be the map sending $(T_k)$ to the symbol $\sum \hbar^\ell f_\ell$.
Recall the algebra $\mathcal{T}$ of Berezin-Toeplitz operators and the symbol map $\si : \mathcal{T} \rightarrow \Ci (M) [[\hbar ]]$ defined in Section \ref{sec:symb-bergm-kern}.  
\begin{thm} \label{thm:BT_analytic}
  $ $
  
  \begin{enumerate}
  \item $\mathcal{T}^\omega$ is a subspace of $\mathcal{T}$ and $\si :
    \mathcal{T}^\omega \rightarrow \mathcal{S}^\omega$ is the restriction of
    $\si : \mathcal{T} \rightarrow \Ci (M) [[\hbar]]$. 
  \item $\si: \mathcal{T}^\omega \rightarrow \mathcal{S}^\omega$ is surjective
    and its kernel consists of the families $(T_k)  \in  \prod_k \op{End}(\Hilb_k)$ such that $\| T_k \| = \bigo ( e^{-k/C})$ for some $C>0$.
  \item  $\mathcal{T}^{\om} $ is closed under product, the corresponding
    product of symbols of $\mathcal{S}^\om$ is $\star$.
    \item 
 $(\op{id}_{\Hilb_k})$ belongs to $\mathcal{T}^{\omega}$, with symbol $\sum
 \hbar^{\ell} \rho_\ell$.
\end{enumerate}
\end{thm}
Since the Schwartz kernel of $\op{id}_{\Hilb_k}$ is the Bergman kernel $\Pi_k$,  the
last assertion is equivalent to the fact $|\Pi_k (x, y_c) |$  is a $\bigo (
e^{-k/C_K} )$ on any compact set $K$ of $M\times \con M$ not intersecting the
  diagonal and that
  $$  \Pi_k  (x,y_c) = E_k (x,y_c) \sum_{\ell =0 } ^{ E ( \ep k)
  } k^{-\ell}  \tilde{\rho}_{\ell} (x,y_c)  + \bigo ( e^{-k/C}) .$$

As already mentioned in the introduction, the third assertion has been proved in \cite{Del}, the last one in \cite{RoSjVu}. The proof of the third assertion relies on an analytic version of the stationary phase lemma \cite{SjAs}. The fact that $\mathcal T$ is closed under product was already proved in \cite{oim} by using a  stationary phase lemma for smooth complex valued phase. The last assertion can certainly be deduced from Theorem \ref{th:estimate_bergman} by
  using the method of \cite{BeBeSj}. Actually, it is a simple corollary of
 the third assertion as was noticed in \cite{Del}. We will reproduce this short proof here.

\subsection{Proof of Theorem \ref{thm:BT_analytic}} \label{sec:proof-theor-refthm:b}

By Proposition \ref{prop:symbole_analytique}, $ \tilde f(\ep, k ) :=  \sum_{\ell =0 } ^{ \lfloor \ep k \rfloor
} k^{-\ell}  \tilde{f}_{\ell}$ has an asymptotic expansion
$$ \tilde f (\ep , k )  = \sum_{\ell =0 } ^{ N
 } k^{-\ell}  \tilde{f}_{\ell} + \bigo ( k^{-N-1}), \qquad \forall N$$
This implies immediately that $\mathcal{T}^\omega \subset
\mathcal{T}$ and that the definitions of the symbols are the same for
Berezin-Toeplitz operators and analytic Berezin-Toeplitz operators. 
Another basic property of the partial sums $\tilde f ( \ep, k)$ is that for $\ep' < \ep$, $f(\ep' , k ) = f( \ep ,k) + \bigo ( e^{-k / C'})$ for some $C'>0$, cf. Proposition \ref{prop:symbole_analytique}. This simple  remark is actually needed to check that $\mathcal{T}^\omega$ is a linear
subspace of $\prod_k \op{End} \Hilb_k$.

To prove that $\si : \mathcal{T}^\omega \rightarrow \mathcal{S}^\omega$ is
surjective, we define for an analytic symbol $\sum \hbar^\ell f_{\ell}$ as
above the kernel
$$ T_k(x,y_c) := \rho (x,y_c) E_{k} (x,y_c) \ti f (\ep,k ) (x,y_c) $$
where $\rho \in \Ci_0 ( W)$ is equal to $1$ on a neighborhood of $\Delta_M$.
The family $(T_k)$ certainly satisfies \eqref{eq:def_BT_analytic} and also
$|T_k | = \bigo ( e^{-k/C_K})$ on any compact subset $K$ of $M \times \con M$
not intersecting $\Delta_M$. But $T_k$ is not the Schwartz kernel of an
endomorphism of $\Hilb_k$ because $T_k$ is not holomorphic. Actually, since $\con
\partial T_k = (\con \partial \rho ) E_k f(\ep,k)$, we have that $|\con
\partial T_k | = \bigo ( e^{-k/C'})$ uniformly on $M \times \con M$ for some
$C'>0$. By the Kodaira-Nakano-H\"ormander inequality, this has the consequence
that the orthogonal projection $S_k$ of $T_k$ onto the space of holomorphic
sections of $(L^k \otimes L') \boxtimes (\con{L}^k \otimes \con{L}')$
is equal to $T_k$ up to a uniform $\bigo ( e^{-k/C'})$ with a larger $C'$. So
the corresponding family of operators is an analytic Berezin-Toeplitz operator
with symbol $\sum \hbar^\ell f_\ell$. 

The kernel of $\si : \mathcal{T}^\omega \rightarrow \mathcal{S}^\omega$
consists of operators whose Schwartz kernel is uniformly in  $\bigo (
e^{-k/C})$ for some $C>0$. This condition is equivalent to the fact that the
operator norm $\| T_k \|$ is in   $\bigo (
e^{-k/C'})$ for some $C'>0$. To prove this, we can in one direction use the Schur
criterion and on the other direction that $$ T_k (x,\con y) = \langle
T_k \xi_{x,k},  
\xi_{y,k} \rangle, \qquad \| \xi_{x,k} \| , \| \xi_{y,k} \| = \bigo (
k^{n/2}) $$
where $\xi_{x,k}$, $\xi_{y,k}$ are the coherent states of
$L^k \otimes L'$ at $x$ and $y$ respectively.

We now prove that $\mathcal{T}^{\omega}$ is closed under product. First
observe that if $(T_k), (S_k) \in \prod_k \op{End} \Hilb_k $ satisfy both the
condition {\em i)} of Definition \ref{def:bertoepan} and their Schwartz kernels are uniformly bounded independently of $k$, then the
same holds for $(T_k S_k)$. Furthermore, for any open subsets $V$ and
$W$ of $M$  such that $V$ has a compact closure
contained in $W$, we have
$$ (T_k S_k ) (x,x_c) = \int_{W} T_k (x,\con{y} ) S_k (y, x_c) \; d\mu (y) +
\bigo ( e^{ - k /C}), $$
for any $(x,x_c) \in V \times \con V $
with a uniform $\bigo$.

Let us assume from now on that $(T_k)$ and $(S_k)$ are both analytic
Berezin-Toeplitz operators with symbols $\sum \hbar^\ell f_{\ell}$ and $\sum
\hbar^\ell g_{\ell}$. By Theorem \ref{th:local_expression} and Theorem \ref{thm:BT_analytic}, we
already know that the restriction to the diagonal of the Schwartz kernel of  $T_k \circ S_k$ has an asymptotic
expansion 
$$ (T_k \circ S_k)  (x,\con{x}) = \sum_{\ell =0 }^N k^{-\ell} h_{\ell} (x) + \bigo (
k^{-N-1}), \qquad \forall N$$
where $\sum \hbar^\ell h_{\ell}$ is the analytic symbol $\sum \hbar^\ell f_\ell
 \star  \sum \hbar^{\ell} g_{\ell}$. So we have to prove that $T_k \circ S_k = E_k
 \tilde{h} ( \ep, k ) + \bigo ( e^{-k/C})$ on a neighborhood of the diagonal.

Introduce the same local data $U \ni x_0$, $\varphi$, $\varphi'$  as in Section \ref{sec:product-formula}.
Let
\begin{gather*}
\wt{\psi} ( x, y_c, y, x_c) = - \wt \varphi ( x, y_c) -  \wt \varphi ( y ,
x_c) +  \wt \varphi ( x, x_c) + \wt \varphi ( y, y_c) \\
\wt {\psi}' ( x, y_c, y, x_c) = - \wt \varphi' ( x, y_c) -  \wt \varphi' ( y ,
x_c) +  \wt \varphi' ( x, x_c) + \wt \varphi' ( y, y_c)
\end{gather*}
If $\tilde f$, $ \tilde g$ are two holomorphic functions of $\wt U = U \times  \con U$, we set
$$ I_k (\tilde f , \tilde g)(x, x_c) = \biggl( \frac{2 \pi }{k} \biggr)^n \int_{B} e^{ -k \wt {\psi} (x, \con y, y,
  x_c) - \wt \psi ' (x, \con y, y, x_c)} \ti f (x,\con y) \ti g (y, x_c) d \mu (y) $$ 
where $B$ is a compact neighborhood of $x_0$ in $U$.

On a neighborhood of $(x_0, \con x_0)$, the Schwartz kernel of $T_k \circ S_k$ is equal to $I_k(\tilde f (\ep,k), \tilde{g} ( \ep, k ) ) E_k  + \bigo (e^{- k /C})$.  It will from Lemma \ref{lem:1} and Lemma \ref{lem:2} that  $I_k(\tilde f (\ep,k), \tilde{g} ( \ep, k ) ) = \tilde h ( \ep, k ) + \bigo ( e^{-k / C})$, which ends the proof that $(T_k \circ S_k)$ is an analytic Berezin-Toeplitz operator.

\begin{lm} \label{lem:1}
  There exist a neighborhood $V$ of $x_0$ in $U$ and $\ep>0$ such that
\begin{gather} \label{eq:expansion}
  I_k ( \tilde{f}, \tilde{g} ) (x,x_c) = \sum_{\ell = 0 } ^{\lfloor  \ep k \rfloor} k^{- \ell} \wt A_{\ell} (\tilde f, \tilde g)  (x,x_c) + \bigo ( e^{-\ep k } |\tilde{f}|_\infty | \tilde g |_\infty) 
\end{gather}
    for any $(x,x_c) \in V \times \con{V}$, with a uniform $\bigo$ and $ |\tilde f |_{\infty}$, $| \tilde g |_{\infty}$ the sup norms  of $\tilde f $ and $ \tilde g$ on $\wt  U$.  
\end{lm}

\begin{lm} \label{lem:2}
  $  \sum_{\ell = 0 } ^{\lfloor \ep k \rfloor } k^{- \ell} \wt A_{\ell} (\tilde f (\ep, k ) , \tilde g ( \ep, k ) ) = \tilde h ( \ep, k ) + \bigo ( e^{-k / C})$.
\end{lm}

\begin{proof}[Proof of Lemma \ref{lem:1}]
This is a consequence of the analytic version of stationary phase lemma proved in \cite{SjAs}.
When $x_c = \con{x}$, the phase $y \rightarrow \wt \psi ( x, \con{y}, y,
\con{x})$ has a critical point at $y = x$ as was used in Section \ref{sec:product-formula} to establish Theorem \ref{th:local_expression}. In the case where
$x_c \neq \con{x}$, we consider the holomorphic extension $(y,
y_c) \rightarrow \wt \psi ( x, y_c, y, x_c)$. As it was already noticed in
\cite{oim}, the critical point is now at $y=x$, $y_c = x_c$. Indeed, we compute
easily from the definition of $\wt{\psi}$ that 
\begin{gather} \label{eq:Tayl_exp}
\wt \psi ( x, x_c + \con{u}, x+u, x_c) = \sum_{i,j} \frac{\partial^2
  \wt{\varphi}}{\partial x_i \partial x_{c,j}} (x,x_c) u_i \con{u}_j + \bigo ( |u|^3). 
\end{gather}
Furthermore, $(\partial^2
  \wt{\varphi}/\partial x_i \partial x_{c,j}) (x,\con{x}) = G_{i,j} (x)$ is
  definite positive. In this situation, we can apply \cite[Theorem 2.8]{SjAs}.
  The first step is a deformation which can be made explicit in
  our case. We identify $U$ with an open set of $\C^n$, $x_0$ being sent to the origin. We denote by $B_r$ the closed ball with radius $r$ of $\C^n$. We choose an open neighborhood $\wt V$ of the origin in $U \times \con{U} = \wt U$ and $r>0$ such that 
  \begin{enumerate}
   \item $(x,x_c) \in \wt V$, $u \in B_r$,
and $t \in [0,1]$ implies that $(tx +u, tx_c + \con {u}) \in \wt U $.
  \item   for any $(x,x_c) \in \wt V$,  $\op{Re} \wt \psi (x,x_c, 1, u) >0$ if $u \in B_r \setminus
    \{ 0 \}$,
    \item for any $(x,x_c) \in \wt V$,  $\op{Re} \wt \psi (x,x_c, t, u) >0$ if $u
      \in \partial B_r $ and $t \in [0,1]$. 
  \end{enumerate}
  
    Then, by this last condition and Stokes Lemma,  for any
  holomorphic $b: \wt U \rightarrow \C$, the integral  
  $$ J_k(t, b)(x,x_c) := \int_{B_r} e^{ - k \wt \psi (x, t x_c +\con{u} ,t x+  u, x_c)
  } \tilde{b}(tx + u,t x_c+ \con{u})\;
  d\mu_L (u),$$
is independent of $t \in [0,1]$ up to a term in $\bigo ( e^{-k/C} \sup_{\wt U}
|b|)$ uniformly with respect to $(x,x_c) \in \wt V$.
The second step  is to prove that  
$$ J_k (1,b) (x,x_c) = \sum_{\ell = 0 }^{E(\ep k)}  k^{-\ell} c_{\ell} (x,x_c)
+ \bigo ( e^{-k/C} \sup_{\wt U} |b|) $$
with a $\bigo$ uniform for $(x,x_c) \in \wt V$. 
This follows from a holomorphic version of Morse Lemma and a careful
application of Laplace method. The  assumption for this second step is the condition 2 above and the fact that the Hessian $\partial^2
  \wt{\varphi} / \partial x_i \partial x_{c,j}  $ in (\ref{eq:Tayl_exp}) has a positive definite real part.  
 This concludes the proof except for the computation of the coefficients in (\ref{eq:expansion}). Actually we already know these coefficients for $x_c = \con x$, and by \cite[Remarque 2.10]{SjAs}, they depend holomorphically on $(x,x_c)$.
\end{proof}
\begin{proof}[Proof of Lemma \ref{lem:2}] This follows from the second part of Lemma \ref{lem:ledernier} and the continuity of the $\wt A_{\ell}$ established in the proof of Proposition \ref{prop:an_symb_alg}.
\end{proof}

We now prove the fourth assertion of Theorem \ref{thm:BT_analytic} by following \cite{Del}. Choose $(T_k)$ in $\mathcal{T}^\omega$ with symbol $\sum \hbar^\ell b_{\ell}$. Replacing $T_k$ by $\frac{1}{2} (T_k + T_k^*)$, we can assume that $T_k$ is self-adjoint. Since $(T_k)$ belongs to $\mathcal{T}$ and has the same symbol of $(\op{id}_{\Hilb_k} )$, we already know that $T_k = \id_{\Hilb_k} + \bigo ( k^{-\infty})$. Our goal is to prove that $T_k = \op{id}_{\Hilb_k} + \bigo ( e^{-k/C} )$. The symbol of $(T_k)$ is idempotent, so $T_k^2 = T_k + \bigo ( e^{-k/C})$, so the spectrum of $T_k$ splits in two parts, concentrating at $0$ and $1$ respectively, more precisely  
$$ \op{sp} (T_k) \subset [-Ce^{-k/C}, Ce^{-k/C}] \cup [1 -Ce^{-k/C} , 1+ Ce^{-k/C} ]$$
with a larger $C$. If $k$ is sufficiently large, these two intervals are
disjoint. Let $n_k$ be the number of eigenvalues counted with multiplicity in
the second interval. Then $\op{tr} T_k = n_k + \bigo ( e^{-k/C})$. The fact
that $T_k = \id_{\Hilb_k} + \bigo ( k^{-\infty})$ implies that
$$\op{tr} T_k  = \dim \Hilb_k + \bigo (
k^{-1}),$$
so $n_k = \dim \Hilb_k$ when $k$ is sufficiently large, so  $T_k  = \op{id}_{\Hilb_k} + \bigo (e^{-k/C})$ as was to be proved.

\section{Multiplier} \label{sec:multiplicator}

For $f \in \Ci (M)$, we define the operator $T_k (f) : \Hilb_k \rightarrow \Hilb_k$ sending $\psi$ into $\Pi_k ( f \psi_k)$.
By \cite{oim}, any Berezin-Toeplitz operator $(S_k)$ has an expansion of the form 
\begin{gather} \label{eq:3} 
S_k = \sum_{\ell = 0}^{N} k^{-\ell} T_k (f_{\ell}) + \bigo(k^{-N-1}), \qquad \forall N.
\end{gather}
where $(f_\ell)$ is a sequence of $\Ci (M)$, and the $\bigo$ is in uniform norm. Conversely, for any sequence $(f_{\ell})$, there exists a Toeplitz operator $S_k$ satisfying (\ref{eq:3}). Furthermore,
$$ \si (S_k) = \sum_{\ell, m } \hbar^{\ell+m} B_{\ell} (f_m)$$
where the $B_{\ell}$ are differential operators of $\Ci ( M)$, $B_0$ being the identity. The map sending $\sum \hbar^{\ell} f_{\ell}$ into the symbol of $(S_k)$ is an isomorphism of $\Ci (M) [[\hbar]]$. 

\begin{thm} \label{th:multiplier}
  $ $
  \begin{enumerate}
    \item 
      For any analytic symbol $\sum \hbar^{\ell} f_{\ell}$ and $\ep >0$ sufficiently small, the family  $ \bigl( T_k (\sum_{\ell =0 }^{\lfloor \ep k \rfloor } k^{-\ell} f_{\ell} ) \bigr) $ is an analytic Berezin-Toeplitz operator.
    \item  For any $f \in \mathcal{C}^{\omega} (M)$, $B_{\ell} (f)$ has a holomorphic extension $\wt B_{\ell} ( \tilde f )$ given locally by
      \begin{xalignat}{2}
        \begin{split} \label{eq:4}  
  &  \wt B_{\ell} ( \tilde f ) (x,x_c) =  \\
  & \sum_{\substack{ p+q+r = \ell \\ |\al | \leqslant p, |\be | \leqslant p}} \frac{   \tilde a_{p, \al , \be } ( x,x_c) }{\al! \be ! (p!)^{-1}}  \partial^\al_y \partial^\be_{y_c}  \bigl( \tilde{\rho}_{q} (x, y_c) f (y, y_c) \tilde{\rho}_{r} (y, x_c)\bigr)  \Bigr|_{\substack{y= x,\\ y_c = x_c}}  
        \end{split}
\end{xalignat}
  where the $a_{p, \al, \be}$'s are the functions introduced in Lemma \ref{lem:maj_ordre_derivee}.
\item  The map sending $ \sum \hbar^m f_m$ into $\sum \hbar^{\ell + m } B_{\ell} (f_m)$ is an isomorphism of $\mathcal{S}^\omega$. Consequently, any analytic Berezin-Toeplitz operator $(S_k)$ has the form
  $$ S_k = T_k \biggl(\sum_{\ell =0 }^{\lfloor \ep k \rfloor } k^{-\ell} f_{\ell} \biggr) + \bigo ( e^{-k/C}) $$
  for an analytic symbol $\sum \hbar^{\ell} f_{\ell}$. 
  \end{enumerate}
\end{thm}

The proof is long, but actually a variation of what we did before. 
\begin{proof}
  1 and 2. We compute the Schwartz kernel of the product
  $$ (\Pi_k f \Pi_k) (x,x_c) = \int_M \Pi_k (x , \con y ) f(y) \Pi_k ( y, x_c) \; d\mu (y) $$
  exactly as we did for the product of two Toeplitz operators in Section \ref{sec:proof-theor-refthm:b}. The only change is the factor $f$. The computation of the symbol is the same as in Theorem \ref{th:local_expression}, where we replace $d$ by the series $\sum_{q,r} \hbar^{q+r} \tilde{\rho}_q ( x, \con x + \con u) f( x+u, \con x + \con u ) \tilde{\rho}_r (x+u, \con x)$. This leads to the formula of $\wt B_{\ell} (f)$. The estimates of $a_{\ell, \al, \be}$ given in Theorem \ref{th:main_estimate} and the fact that $\sum \hbar^\ell \rho_{\ell}$ is analytic, Theorem \ref{th:estimate_bergman}, imply that the new symbol $\sum \hbar^{\ell + m } B_\ell ( f_m)$ is analytic when $\sum \hbar^m f_m$ is.

  3. We already know that $B= \sum \hbar^{\ell }B_{\ell}$ sends $\mathcal{S}^{\omega}$ into itself. We have to prove that the same holds for $B^{-1}$. It is as difficult as proving that $\sum \hbar^{\ell} \rho_{\ell}$ is analytic. Fortunately, we can follow the same method. First, we deduce from Theorem \ref{th:main_estimate} and Theorem \ref{th:estimate_bergman} that
  $$ \| \wt {B}_{\ell} \|_{t,s} \leqslant \frac{ (C \ell)^\ell}{ (t-s)^{\ell}} .$$
  The proof is the same as the one of (\ref{eq:step1}) except that we now use Formula (\ref{eq:4}). We deduce that the inverse $B^{-1} = \sum \hbar^\ell C_{\ell}$ satisfies
  $$ \| \wt {C}_{\ell} \|_{t,s} \leqslant 2^{\ell-1} \frac{ (C \ell)^\ell}{ (t-s)^{\ell}} $$
  with exactly the same proof as in Theorem \ref{th:estimate_bergman}. We deduce that $\sum \hbar^{\ell} C_{\ell}$ sends $\mathcal{S}^\omega$ into itself.  
\end{proof}

We can also prove that the differential operators $B_{\ell}$ have locally the form
$$B_{\ell} = \sum_{|\al|, |\be|\leqslant \ell} b_{\ell, \al , \be} \partial^{\al}\con{\partial}^{\be}$$ with analytic coefficients admitting holomorphic extensions on $U \times \con U$ such that on any compact set $|\tilde{b}_{\ell, \al , \be} | \leqslant C^{\ell +1} \ell !$. 
The coefficients $C_{\ell}$ of $B^{-1} = \sum \hbar^\ell C_\ell$ have exactly the same property, cf. Section \ref{sec:estim-holom-star}.

\section{Appendix}

In the first part, we discuss the asymptotic expansion of analytic symbols. We work
in an abstract setting where the symbols are not functions but belong to a
normed space, because it makes the discussion simpler. One goal is to compare
some remainder estimates \eqref{eq:dasan} coming from \cite{BoKr} with
the partial sums \eqref{eq:partial_Sum_symbol} introduced in \cite{SjAs}.
Even if we haven't found such a discussion in the literature, we do not claim
that these
results are original. 

The second part is a digression on the method we use to prove Theorem
\ref{th:estimate_bergman} and Theorem \ref{th:multiplier}. We propose
alternative arguments and some generalisation. 

\subsection{Analytic asymptotic expansion}
Let $(E, | \cdot |)$ be a normed space. Consider a sequence $(u(k))$ of $E$ having an asymptotic expansion of the form
\begin{gather} \label{eq:das}
u(k) = \sum_{\ell =0 }^{N-1} a_\ell k^{-\ell} + \bigo ( k^{-N}) , \qquad \forall N
\end{gather}
with coefficients $a_\ell \in E$. 
Two important facts are that $(u(k))$ is determined modulo $\bigo ( k^{-\infty})$ by the $a_\ell$'s, and for any sequence $(a_{\ell})$, there exists a sequence $(u(k))$ satisfying   (\ref{eq:das}).

We are interested in a particular class of asymptotic expansions where the remainder in (\ref{eq:das}) has the following explicit upper bound
\begin{gather} \label{eq:dasan} 
\bigl| u(k) - \sum _{\ell =0 }^{N-1} a_\ell k^{-\ell} \bigr| \leqslant C^{N+1} k^{-N} N! , \qquad \forall \; k\in \N^*,\;  N \in \N 
\end{gather}
for a constant $C$ independent of $k$ and $N$. Unlike the expansion
\eqref{eq:das}, the sequence $(u(k))$ is uniquely determined up to a
$\bigo (e^{-\ep k})$ by \eqref{eq:dasan}. Furthermore, the coefficients
$a_{\ell}$ have a particular growth. The precise result is as follows. 

\begin{prop} \label{prop:symbole_analytique} $ $
  \begin{enumerate}
  \item If a sequence $(u(k))$ satisfies (\ref{eq:dasan}), then the coefficients $(a_\ell)$ satisfy
    $ |a_\ell | \leqslant C^{\ell+1} \ell !$ with the same constant $C$.
  \item Assume that $(u(k))$ satisfies (\ref{eq:dasan}). Then $(u'(k))$ satisfies  (\ref{eq:dasan})  with the same coefficients $a_\ell$ and possibly a larger constant $C$ if and only if  there exists $\ep>0$ such that $u (k ) = u'( k) + \bigo ( e^{-\ep k})$.
  \item If  $|a_{\ell} | \leqslant (C')^{\ell+1} \ell! $ for $C'>0$ and
    $\ep>0$ is such that $\ep C' <1$, then
\begin{gather} \label{eq:partial_Sum_symbol}
  u (k) := \sum_{\ell =0 }^{\lfloor \ep k  \rfloor} a_{\ell} k^{-\ell} 
\end{gather}
satisfies (\ref{eq:dasan}) for some $C>0$.
  \end{enumerate}
\end{prop}

\begin{proof} 1. (\ref{eq:dasan}) implies that $|a_N k^{-N} | \leqslant C^{N+1} k^{-N} N! + C^{N+2} k^{-N-1} (N+1)!$. Multiplying by $k^N$ and taking the limit $k \rightarrow \infty$, we get $|a_N | \leqslant C^{N+1}  N!$.

  2. Assume that $(u(k))$ and $(u'(k))$ satisfy both (\ref{eq:dasan}). Then
  $$| u (k) - u'(k) | \leqslant 2C^{N+1} k^{-N} N! \leqslant 2 C (CN/k)^N.$$  
  Choose $\ep$ such that $C \ep <1$ and set $N =\lfloor \ep k  \rfloor$. Then $ N \leqslant \ep k$ so $CN/k \leqslant C \ep$ so $(CN/k)^N \leqslant (C \ep)^N \leqslant (C\ep )^{\ep k - 1 }$ because $N \geqslant \ep k -1 $. So $|u (k) - u'(k)| \leqslant \ep^{-1} \exp ( \ep k \ln ( C \ep ) )$. Hence $u (k) = u'(k) +\bigo ( e^{-\ep' k})$ with $\ep' = - \ep \ln (C\ep)$.
  Conversely, we have to prove that for any $\ep$, there exists $C >0$ such that $e^{-\ep k } \leqslant C (CN/k)^N$ for any $k \in \N^*$ and $N\in \N$. The function $x \rightarrow \ln x - \epsilon x$ is bounded above on $\R_{>0}$, so
  $$\ln (k/N) - \ep (k/N) \leqslant \ln C, \qquad \forall \; k,N \in \N^*$$ if $C$ is sufficiently large. Multiplying with $N$ and taking the exponential, we get $e^{-\ep k} \leqslant (CN/k)^N$. We conclude easily.

  3. By assumption $|k^{-\ell } a_{\ell} | \leqslant C' (C'/k)^\ell \ell ! =: b_{\ell}$. We will use that
  $$ \frac{b_\ell} {b_{\ell -1}} = C' \ell / k.$$
  Let $N' =\lfloor \ep k \rfloor$.
  Assume that $N \leqslant \ell \leqslant N'$.  Then $\ell \leqslant \ep k$, so $b_{\ell} / b_{\ell-1} \leqslant C'\ep <1$, so
  $$ \sum_{\ell = N}^{N'} b_{\ell} \leqslant b_N ( 1 + \ldots + (C'\ep)^{N' - N} ) \leqslant \frac{b_N}{1 - C'\ep} = \frac{C'}{1 - C' \ep} (C'/k)^N N! $$
  Assume now that $N' < \ell  \leqslant N$. Then $\ep k \leqslant \ell$, so
  $b_{\ell }/ b_{\ell +1 } \leqslant (C'\ep)^{-1} =: r$. Since $r>1$, 
  \begin{gather*}  \sum_{\ell = N'+1}^{N-1} b_{\ell} \leqslant  \sum_{\ell = N'+1}^{N} b_{\ell} \leqslant b_N ( 1 + r + \ldots + r^{N -N' -1} ) \\ \leqslant b_N \frac{ r^{N- N'}}{ r -1 } \leqslant b_N \frac{r^N}{r-1} = \frac{C'}{r-1} (\ep k )^{-N} N!  
  \end{gather*}
  which concludes the proof.
\end{proof}

Let us call a formal series $a = \sum \hbar^\ell a_{\ell}$ of $E[[\hbar]]$ an {\em analytic symbol} if $|a_{\ell}| \leqslant C^{\ell+1} \ell!$ for some $C>0$. For any $\ep >0$, we set $a(\ep,k) := \sum_{\ell =0 }^{\lfloor \ep k \rfloor } k^{-\ell} a_\ell$. 
Choose a second normed space $(E', | \cdot | ')$ and let $\mathcal{L} ( E, E')$ (resp. $\mathcal{B} (E,E')$)  be the space of bounded linear maps $E \rightarrow E'$ (resp. bounded bilinear maps $E \times E \rightarrow E'$) with its natural norm.

\begin{lm} \label{lem:ledernier}
  $ $
  \begin{enumerate}
    \item 
      For any analytic symbols $a \in E[[\hbar ]]$ and $P \in \mathcal{L} (E, E') [[\hbar]]$, $b = \sum_{\ell, m } \hbar ^{\ell + m } P_{\ell} ( a_m)$ is an analytic symbol of $E'[[\hbar]]$. Furthermore, if $\ep>0$ is sufficiently small, then there exists $C>0$ such that 
      $$  P(\ep,k ) (a ( \ep,k ) ) = b ( \ep,k) + \bigo ( e^{- k/C} ).$$
      \item  For any analytic symbols $a, a' \in E[[\hbar ]]$ and $B \in \mathcal{B} (E, E') [[\hbar]]$, $b = \sum_{\ell, m,p } \hbar ^{\ell + m+p } B_{\ell} ( a_m, a'_p)$ is an analytic symbol of $E'[[\hbar]]$. Furthermore if $\ep>0$ is sufficiently small, then there exists $C>0$ such that 
      $$  B(\ep,k ) (a ( \ep,k ), a' ( \ep, k) ) = b ( \ep,k) + \bigo ( e^{-  k/C} ).$$
  \end{enumerate}
\end{lm}

\begin{proof}
  Assume that $\| P_{\ell} \| \leqslant C^{\ell+1} \ell !$ and $| a_{m} | \leqslant C^{m+1} m!$. Then $| P_{\ell} a_{m} | \leqslant C^2 C^{\ell+m } (\ell + m )!$, so $| b_p | \leqslant (p+1) C^{2 + p} p! $ and we conclude that $b$ is analytic. Furthermore, for $N = \lfloor \ep k \rfloor$, we have
  $$ \bigl| P(\ep,k ) (a ( \ep,k ) ) - b ( \ep,k) \bigr| \leqslant \sum ^{\ell \leqslant N, \; m \leqslant N}_{N< \ell+ m} k^{-\ell -m} |P_\ell (a_m)| \leqslant N \sum_{p = N+1}^{2N} C^2 (Cp/k)^p  $$
  by the previous estimate. 
  $ N < p\leqslant 2N$ implies that $\ep k \leqslant p \leqslant 2 \ep k$ so that $(Cp/k)^p \leqslant (2 C \ep)^p \leqslant (2 C\ep )^{\ep k}$ where we have assumed that $2C \epsilon <1$. It follows that
  $$ \bigl| P(\ep,k ) (a ( \ep,k ) ) - b ( \ep,k) \bigr| \leqslant N^2 C^2 (2 C \ep )^{\ep k} \leqslant (\ep C) ^2 k^2 (2 C\ep)^{\ep k} = \bigo (e^{-k/C'}) $$
  for a sufficiently large $C'$. The proof of the second part is similar. 
\end{proof}

\subsection{Estimates for holomorphic star-products} \label{sec:estim-holom-star}

Let $\op{End} (E)$ be the algebra of endomorphisms of a vector space $E$. Let $(\| \cdot \|_\ell, \ell \in \N)$ be a family of seminorms of $\op{End} E$ satisfying
\begin{gather} \label{eq:submult} 
\| P \circ Q \| _{p+q} \leqslant \| P \| _p  \| Q\|_q , \qquad \forall P, Q \in \op{End} E
\end{gather}
for any $p, q \in \N$.
Consider a formal series $\op{id} - \sum_{\ell \geqslant 1} \hbar^{\ell} P_\ell$ of $(\op{End} E)[[\hbar]]$ with inverse $\op{id} + \sum_{\ell \geqslant 1} \hbar^{\ell} Q_\ell$.
\begin{lm} \label{lem:un_autre}
  If there exists $C>0$ such that $\| P_{\ell} \|_{\ell} \leqslant C^{\ell }$ for any $\ell$, then there exists $C'>0$ such that $\| Q_{\ell} \|_{\ell} \leqslant (C')^{\ell}$ for any $\ell$.
\end{lm}
We give two proofs, the first one is a direct generalization of the proof of
Theorem \ref{th:estimate_bergman}, the second is inspired from \cite{BoKr}, \cite{SjAs}.

\begin{proof}
 A first proof is to establish the formula
\begin{gather} \label{eq:8} 
 Q_{m} = \sum _{\substack{\ell \geqslant 1, \; (i_1, \ldots, i_{\ell}) \in \Z_{>0}^\ell,\\ i_1 + \ldots + i_\ell = m} } P_{i_1} \ldots P_{i_{\ell}} 
\end{gather}
 and we conclude easily by using that the number of terms in the sum is $2^{m-1}$.
 Another proof less precise but interesting as well is to introduce for any formal series $R= \sum \hbar^\ell R_{\ell}$ and $\rho >0$ the series $f( R, \rho) = \sum \rho^{\ell} \| R_{\ell} \|_{\ell} $. Then
 \begin{enumerate}
   \item 
     $f(R,\rho)$ converges for some $\rho >0$ if and only there exists $C>0$ such that $\| R_{\ell} \|_{\ell} \leqslant C^{\ell +1}$ for any $\ell$.
   \item $f(RS, \rho) \leqslant f( R, \rho) f(S, \rho)$ by  (\ref{eq:submult}).
   \item if $(R^{(n)})_n$ is a sequence of $(\op{End} E)[[\hbar]]$ such that $R^{(n)} = \bigo ( \hbar^n)$ for any $n$, then $f ( \sum_n R^{(n)} , \rho ) \leqslant \sum_n f(R^{(n)} , \rho)$ by triangle inequality.
 \end{enumerate}
Then we can argue as follows. Set $R = \sum_{\ell \geqslant 1} \hbar^{\ell} P_{\ell}$. The assumption on the $P_{\ell}$'s implies that $f ( R, \rho) = \bigo ( \rho)$, so $f( R, \rho) \leqslant \delta <1$ when $\rho$ is sufficiently small. Applying the previous properties we have
 $$f( \sum R^n, \rho) \leqslant \sum f( R^n, \rho) \leqslant \sum \delta^n < \infty $$
 which concludes the proof because $\op{id} + \sum \hbar^{\ell} Q_{\ell} = \sum R^n$. 
\end{proof}

We apply this to holomorphic differential operators of an open set $\Om $ of $\C^n$.
Consider a formal series $\op{id} - \sum_{\ell \geqslant 1} \hbar^{\ell} P_\ell$ where for any $\ell$,
$$P_{\ell} = \sum _{|\al|\leqslant  N \ell } a_{\ell, \al} \frac{1}{\al!} \partial^\al $$
and the $a_{\al, \ell} $'s are holomorphic functions of $\Om$. Here $N$ is any positive integer, $N=1$ or $2$ in our applications. Then the inverse $\op{id} + \sum_{\ell \geqslant 1} \hbar^{\ell} Q_\ell$ has the same form $ Q_{\ell} = \sum _{|\al|\leqslant  N \ell } b_{\ell, \al} \frac{1}{\al!} \partial^\al $ as follows for instance from (\ref{eq:8}).

\begin{lm} \label{lem:vrai_dernier}
  Assume that on any compact set $K$ of $\Om$, there exists $C_K$ such that $|a_{\ell, \al} | \leqslant C_K^{\ell } \ell^{\ell}$ for any $\al$, $\ell$. Then the family $(b_{\ell, \al}) $ satisfies the same estimates, with different constants $ C_K$.  
\end{lm}

\begin{proof} 
Let $x_0 \in \Om$ and $0<t_0\leqslant 1$ be such that the closure of the ball $B(x_0, t_0)$ is contained in $\Om$. For any $0<s<t < t_0$, define as in Section \ref{sec:main-estimates} $\| P \|_{t,s}$ as the norm of the restriction $P : \mathcal{B} ( B(x_0, t)) \rightarrow \mathcal{B} ( B(x_0, s))$. Set
$$ \| P \|_{\ell} := \sup \{ \ell^{-\ell} (t-s)^\ell \| P \|_{t,s} / \;  0<s<t<t_0 \} .$$
The submultiplicativity (\ref{eq:submult}) follows from Lemma \ref{lem:dertech}.
By the first part of Lemma \ref{lem:dertech} and the assumption on the $a_{\ell,\al}$'s, we have $\| P_{\ell} \|_{\ell} \leqslant C^\ell$. This implies by Lemma \ref{lem:un_autre} that $\|Q_{\ell} \|_{\ell} \leqslant C^\ell$ with a larger $C$. For any $x \in B (x_0, t_0/2)$, the sup norm of $(z - x)^\al$ on $B(x_0,t_0)$ is smaller than $(3 t_0/2)^{|\al|} \leqslant (3/2)^{\ell}$. So
\begin{gather*}
| b_{\ell, \al} (x) | = | Q_{\ell} ( (z -x )^{\al} ) (x) | \leqslant (3/2)^{\ell} \| Q_{\ell} \|_{t_0, \frac{t_0}{2}}   \\ \leqslant(3/2)^{\ell}   \ell^{\ell} (t_0/2)^{-\ell} \| Q_{\ell}  \|_{\ell}  \leqslant (3C /t_0)^{\ell} \ell^{\ell} . 
\end{gather*}
as was to be proved.
\end{proof}

\bibliographystyle{abbrv}
\bibliography{Analytic}

\vspace{1cm}

\noindent
\begin{tabular}{l}
Laurent Charles, \;   Sorbonne Universit\'e, CNRS, \\
 Institut de Math\'{e}matiques  
de Jussieu-Paris Rive Gauche \\
 F-75005 Paris, France. 
\end{tabular}

\end{document}